\documentclass[10pt]{article}
\usepackage{color}
\usepackage{amssymb}
\usepackage{amsthm,array,amssymb,amscd,amsfonts,latexsym, url}
\usepackage{amsmath}

\newtheorem{theo}{Theorem}
\newtheorem{prop}[theo]{Proposition}
\newtheorem{lemm}[theo]{Lemma}
\newtheorem{coro}[theo]{Corollary}
\newtheorem{rema}[theo]{Remark}
\newtheorem{defi}[theo]{Definition}

\begin{document}
\title{On a conjecture of Matsushita }
\author{Bert van Geemen and Claire Voisin\footnote{This research has been   supported by  The Charles Simonyi Fund and The Fernholz Foundation.}} \maketitle \setcounter{section}{-1}
\section{Introduction}
Let $X$ be a smooth projective hyper-K\"ahler manifold of dimension
$2n$ admitting a Lagrangian fibration
$f:X\rightarrow B$. The smooth  fibers $X_b$ of
$f$ are thus abelian varieties of dimension $n$. When $B$ is smooth, it is known (see \cite{oguiso}) that the restriction map
$$H^2(X,\mathbb{Z})\rightarrow H^2(X_b,\mathbb{Z})
$$
has rank $1$, so that the fibers $X_b$ are in fact canonically polarized by the restriction of any
ample line bundle on $X$. Denoting
by $\alpha$ the type
 of the polarization, we thus have a moduli morphism
 $$ m:B^0\rightarrow \mathcal{A}_{n,\alpha}$$
 where $B^0\subset B$ is the open set parameterizing smooth fibers and $\mathcal{A}_{n,\alpha}$
 is the moduli space of $n$-dimensional abelian varieties with a polarization
 of type $\alpha$.
 It has been conjectured by  Matsushita that $m$ is either generically finite on its image
 or  constant (the second case being the case of isotrivial fibrations).
This conjecture was communicated to us by Ljudmila Kamenova and Misha Verbitsky.
Our goal in this note is to prove the following weakened form of Matsushita's conjecture.
Let $P\subset H^2(X,\mathbb{Z})$ be the N\'eron-Severi group of $X$. One can construct
the universal family $\mathcal{M}_P$ of marked deformations of $X$  with fixed Picard group $P$,
that is deformations
$X_t$ for which all the  classes in $P$ remain Hodge on $X_t$.
 It follows from \cite{voisinlag}  that such deformations  locally preserve
the Lagrangian fibrations on $X$. So deformations parameterized
 by $\mathcal{M}_P$ automatically induce a deformation of the triple $(X,f,B)$, at least on a
 dense Zariski open set of $\mathcal{M}_P$.

\begin{theo}\label{theomatsu} Let $X$ be a projective hyper-K\"ahler manifold of dimension
 $2n$ admitting a Lagrangian fibration $f:X\rightarrow B$, where $B$ is smooth. Assume
 $b_{2,tr}(X):=b_2(X)-\rho(X)\geq5$.
 Then the  deformation $(X',f',B')$ of the triple $(X,f,B)$
 parameterized by a very general point of $\mathcal{M}_P$
satisfies Matsushita's conjecture, that is
the moduli map $m':B' \dashrightarrow \mathcal{A}_{n,\alpha}$ is either constant or generically of maximal rank $n$.

\end{theo}
\begin{coro} In the
 space $\mathcal{M}_P$ of   deformations of $X$ with
constant N\'eron-Severi  group, either there is a dense Zariski
 open set of points parameterizing
triples $(X',f',B')$ for which the moduli map has maximal rank $n$, or
for any point of $\mathcal{M}_P$, the moduli map is constant.

\end{coro}
This follows indeed from the fact that the condition that $m$ is generically of maximal rank
is Zariski open.
\begin{rema}{\rm The assumption $b_2(X)-\rho(X)\geq5$ in Theorem
\ref{theomatsu} is presumably not essential here, but
some more arguments would be needed otherwise. It is related to the simplicity of the orthogonal groups. Note also that no compact hyper-K\"ahler manifold with $b_2<5$ is known, so in practice, this does not seem to be very restrictive.}
\end{rema}
\begin{rema}{\rm Concerning the assumption that $B$ is smooth, it is believed that it always holds.
Matsushita \cite{matsu1}, \cite{matsu2} proved a number of results on the geometry and topology of the base
$B$ suggesting that it must be isomorphic to $\mathbb{P}^n$, and Hwang proved this is the case if
it is smooth.
}
\end{rema}

Our proof  will use the fact that the very general
point of $ \mathcal{M}_P$  parameterizes a deformation $X'$ of $X$ for which
the Mumford-Tate group of the Hodge structure on
$$H^2(X',\mathbb{Q})_{tr}=H^2(X',\mathbb{Q})^{\perp P}\cong H^2(X,\mathbb{Q})^{\perp P}$$
is the full special orthogonal group of the $\mathbb{Q}$-vector space
$H^2(X',\mathbb{Q})_{tr}$  equipped
with the Beauville-Bogomolov intersection form (see Section \ref{secMT}).
Theorem \ref{theomatsu} will be then a consequence of  the following more precise result:
\begin{theo} \label{theomatsuMT} Let $X$ be a projective hyper-K\"ahler manifold of dimension
 $2n$ admitting a Lagrangian fibration $f:X\rightarrow B$ with $B$ smooth. Assume
 $b_{2,tr}(X):=b_2(X)-\rho(X)\geq5$ and the Mumford-Tate group of the Hodge structure on
 $H^2(X,\mathbb{Q})$ is the group $SO(H^2(X,\mathbb{Q})_{tr},q)$.
 Then  the pair $(X,f)$ satisfies Matsushita's conjecture.

\end{theo}
The  proof of Theorem \ref{theomatsuMT} will be obtained as a consequence of the following
proposition (cf. Proposition \ref{propvgeemen}) establishing a universal property
of the Kuga-Satake construction (see \cite{kugasatake}, \cite{deligne2}, \cite{vangeemen}):
\begin{prop} Let $(H,q, H^{p,q})$ be a weight $2$ polarized Hodge structure of $K3$ type, that is, such that
$h^{2,0}=1$. Assume that ${\rm dim}\,H\geq 5$ and the Mumford-Tate group of $(H,H^{p,q})$ is the special
orthogonal group of of $(H,q)$. Then
for any irreducible weight $1$ Hodge structure such that, for some weight $1$ Hodge structure
$H_2$, there
is an embedding of weight $2$ Hodge structures
$$H\subset H_1\otimes H_2,$$
$H_1$ is isomorphic to an irreducible weight $1$ sub-Hodge structure of
$H^1(A_{KS}(H),\mathbb{Q})$, where $A_{KS}(H)$ is the Kuga-Satake variety of $(H,q, H^{p,q})$.
\end{prop}
\begin{rema}{\rm This implies  that there
is a finite and in particular discrete set of such Hodge structures
$H_1$. The condition on the Mumford-Tate group of $H$ is  quite essential here.
We will give in the last section an example of a $K3$ type polarized Hodge
 structure $H$ for which  there is  a continuous family of irreducible weight $1$ Hodge structures
$H_1$ satisfying the conditions above.
}
\end{rema}
{\bf Thanks.} {\it The second author would like to thank Ljudmila Kamenova for bringing
Matsushita's
conjecture to her attention and also  for interesting discussions and useful comments on a version of this note.
}

\section{Mumford-Tate groups and the Kuga-Satake construction \label{secMT}}
Let $(H, H^{p,q})$ be a rational Hodge structure of weight $k$. The
group $\mathbb{S}^1$ acts on $H_\mathbb{R}$ by the following rule:
$z\cdot\alpha^{p,q}=z^p\overline{z}^q\alpha^{p,q}$ for $z\in
\mathbb{S}^1$ and $\alpha^{p,q}\in H^{p,q}\subset H_\mathbb{C}$.
\begin{defi} The Mumford-Tate group of $H$ is the smallest algebraic subgroup
of $Gl(H)$ which is  defined over $\mathbb{Q}$ and contains the
image of $\mathbb{S}^1$.

\end{defi}
Let $X$ be a compact hyper-K\"ahler manifold. Consider the Hodge structure
of weight $2$ on $H^2(X,\mathbb{Q})$. It is compatible with the
Beauville-Bogomolov intersection form $q$ (by the first
Hodge-Riemann bilinear relations), so that its Mumford-Tate group is
contained in $SO(q)$. We now have:
\begin{lemm} \label{theoMT} Let $P\subset {\rm NS}(X)\subset H^2(X,\mathbb{Q})$ be a subspace
which contains an ample class (so that the Beauville-Bogomolov form
is nondegenerate of signature $(1,l)$ on $P$). Then for a very
general marked deformation $X'$ of $X$ for which $P\subset {\rm NS}(X')$, the
Mumford-Tate group of the Hodge structure on
$H^2(X',\mathbb{Q})_{tr}$ is the whole special  orthogonal group
$SO(H^2(X',\mathbb{Q})_{tr},q)$.
\end{lemm}
\begin{rema}{\rm Note that the fact that the period map for hyper-K\"ahler manifolds
is open implies that for $X'$ as above, $H^2(X',\mathbb{Q})_{tr}$ is
nothing but the orthogonal complement of $P$ in $H^2(X',\mathbb{Q})$
with respect to $q$.
}
\end{rema}
\begin{proof}[Proof of lemma \ref{theoMT}] Via the period map,  the marked deformations $X_t$ of $X$ for which
$P\subset {\rm NS}(X_t)$ are parameterized by an open set $\mathcal{D}_P^0$ in the period
domain
$$\mathcal{D}_P=\{\sigma_t\in \mathbb{P}(H^2(X,\mathbb{C})^{\perp P}),\,q(\eta_t)=0,\,q(\eta_t,\overline{\eta}_t)>0\}.
$$
For such a period point $\sigma_t$, the Mumford-Tate group
$MT(H^2(X_t,\mathbb{Q}))$ is the subgroup leaving invariant
all the Hodge classes in the induced Hodge structures on the tensor powers
$\bigotimes H^2(X_t,\mathbb{Q})$.
 For each such class $\alpha$, either $\alpha$ remains a
Hodge class everywhere on the family, or the locus where it is  a Hodge class is a
closed proper analytic subset of the period domain. As there are countably
many such Hodge classes, it follows that the Mumford-Tate group
for the very general fiber $X'$ of the family contains the Mumford-Tate groups of
$H^2(X_t,\mathbb{Q})$ for all
 $t\in \mathcal{D}_P^0$.
 We then argue by induction on ${\rm dim}\,H^2(X,\mathbb{Q})^{\perp P}$.
 If ${\rm dim}\,H^2(X,\mathbb{Q})^{\perp P}=2$, then it is immediate to check that
$MT(H^2(X,\mathbb{Q}))$ is the Deligne torus itself, which is equal to
$SO(H^2(X_t,\mathbb{Q})^{\perp P})$. Suppose now that we proved the result
for ${\rm dim}\,H^2(X,\mathbb{Q})^{\perp P}=k-1$ and assume ${\rm dim}\,H^2(X,\mathbb{Q})^{\perp P}=k\geq 3$.
First of all, we easily see that the  strong form of Green's theorem on the density of the Noether-Lefschetz
locus holds, by which we  mean  the following statement:

 {\it There exists a non-empty open set
$V\subset H^2(X,\mathbb{R})^{\perp P}$ such that for any
$\lambda \in V\cap H^2(X,\mathbb{Q})^{\perp P}$, the Noether-Lefschetz locus
$$NL_\lambda\cap \mathcal{D}_P^0=:\{t\in \mathcal{D}_P^0, \,\lambda\in H^{1,1}(X_t)\}=\{t\in \mathcal{D}_P^0, \,q(\sigma_t,\lambda)=0\}$$
is nonempty.}

This is of course a consequence of the  Green density theorem (see \cite[17.3.4]{voisinbook}), but in our case where the period map is open, this is immediate, since letting
 $\sigma\in \mathcal{D}_P^0$ be the period point of $X$, for any open set $U\subset \mathcal{D}_P^0$ containing $\sigma$ and contained in the image of the period map, and
for any $\lambda \in  H^2(X,\mathbb{Q})^{\perp P}$, one has
$U\cap NL_\lambda=\{\sigma_t\in U,\,q(\sigma_t,\lambda)=0\}$, which means equivalently that
$\lambda\in H^2(X,\mathbb{Q})^{\perp <P,\sigma_t>}$ and by taking complex conjugates,
$$\lambda\in H^2(X,\mathbb{Q})^{\perp <P,\sigma_t,\overline{\sigma_t}>}.$$
But clearly, $\cup_{\sigma_t\in U}H^2(X,\mathbb{R})^{\perp <P,\sigma_t,\overline{\sigma_t}>}$ is an
open subset of
 $H^2(X,\mathbb{R})$. We thus can take for $V$ this open set.

For any $t\in NL_\lambda\cap \mathcal{D}_P^0$, the rational subspace
$<P,\lambda>\subset H^2(X_t,\mathbb{Q})$ is contained in ${\rm NS}(X_t)_\mathbb{Q}$
and applying the induction hypothesis, we conclude that for the very general point
$X'_\lambda$ of $NL_\lambda\cap \mathcal{D}_P^0$, the Mumford-Tate group of
$H^2(X'_\lambda,\mathbb{Q})$ is equal to $SO(H^2(X'_\lambda,\mathbb{Q})^{\perp <\lambda,P>},q)$ (and acts as the identity on $<\lambda,P>$).

By the previous argument, we then conclude that for the very general point $X'$ of
$\mathcal{D}_P^0$, the Mumford-Tate group
$MT(H^2(X',\mathbb{Q}))$ contains the orthogonal groups $SO(H^2(X'_\lambda,\mathbb{Q})^{\perp <\lambda,P>},q)$ for any $\lambda \in V\cap H^2(X,\mathbb{Q})^{\perp P}$. As $V$ is open in
$H^2(X,\mathbb{R})^{\perp P}$,
it immediately follows that $MT(H^2(X',\mathbb{Q}))$ is equal to the orthogonal group $SO(H^2(X',\mathbb{Q})^{\perp P},q)$.
\end{proof}

Let now $X$ be a hyper-K\"{a}hler manifold admitting a Lagrangian
fibration $X\rightarrow B$. Let $P:={\rm NS}(X)$. We get the following:

\begin{coro} \label{ledefMT} There exists a (small) deformation $X'$ of $X$
which  is projective with N\'eron-Severi group
 $P$, admits a Lagrangian fibration
$X'\rightarrow B'$ deforming the Lagrangian fibration of $X$,
and such that the Mumford-Tate group of $H^2(X',\mathbb{Q})$ is equal to $SO(H^2(X',\mathbb{Q})^{\perp P},q)$.
\end{coro}
\begin{proof}  By Lemma \ref{theoMT}, the very general $X'$ in
the family $\mathcal{M}_P$ of deformations of $X$ with
N\'eron-Severi group containing  $P$  has  Mumford-Tate group
$SO(H^2(X',\mathbb{Q})^{\perp P},q)$. Furthermore, $X'$ is also projective, at least on a
dense open
set of the deformation family.
On the other hand, it follows from the stability result of
\cite{voisinlag} that  deformations of $X$ preserving
${\rm NS}(X)$ locally preserve any Lagrangian fibration on $X$. So  if the deformation is small
enough, $X'$ admits a Lagrangian fibration deforming the one of $X$.
\end{proof}

Recall \cite{deligne}, \cite{kugasatake}, \cite{vangeemen} that a
polarized integral Hodge structure $H$ of weight $2$ with $h^{2,0}=1$ has an
associated Kuga-Satake variety $A_{KS}(H)$, which is an abelian variety with
the property that the Hodge structure $H$ can be realized (up to a
shift) as a sub-Hodge structure of the weight $0$ Hodge structure on
${\rm End}\,(H^1(A_{KS}(H),\mathbb{Z}))$. If $H$ is a rational polarized
Hodge structure, $A_{KS}(H)$ is  defined only up to isogeny.
The Kuga-Satake variety is essentially
constructed by putting, using the Hodge structure on $H$, a complex
structure on the underlying vector space of the  Clifford algebra
$C(H_\mathbb{R},q)$, which provides a complex structure on the real
torus $C(H_\mathbb{R},q)/C(H)$. In general, the Kuga-Satake is not a
simple abelian variety, because it has a big endomorphism algebra
given by right Clifford multiplication of $C(H)$ on this torus. The main ingredient in our proof
of Theorem \ref{theomatsuMT} will be the
following result:

\begin{prop} \label{propvgeemen} Let $(H,q)$ be a weight $2$
polarized Hodge structure with Mumford-Tate group equal to
$SO(q)$. Let $A,\,B$ be  polarized weight $1$ rational Hodge
structures such that $H\subset A\otimes B$ as weight $2$ Hodge
structures. Then if $A$ is simple (as a Hodge structure) and ${\rm dim}\,H\geq5$, $A$  is
isomorphic  as a rational Hodge structure to $H^1(M,\mathbb{Q})$,
where $M$ is an abelian subvariety of the Kuga-Satake variety of
$H$.
\end{prop}

\begin{proof}
The Mumford-Tate group $MT(A\otimes B)$ maps onto $MT(H)$. As ${\rm dim}\,H\geq5$,
the Lie algebra $mt(H)=so(q)$ is  simple,
so it is a summand of $mt(A\otimes B)$.
As $MT(A\otimes B)\subset MT(A)\times MT(B)$,  the
Lie algebra $mt(A\otimes B)$ is contained in  $mt(A)\times mt(B)$ and
the projection of the simple Lie algebra $mt(H)=so(q)$ to $mt(A)$ and to $mt(B)$ is injective.

If $mt(A\otimes B)$ contains both copies of $so(q)$ then the Mumford Tate group of the tensor product of the corresponding weight one sub-Hodge structures of $A$ and $B$ has  $so(q)\times so(q)$ as Lie algebra.
This contradicts that $A\otimes B$ has a sub-Hodge structure with $h^{2,0}=1$.
Using Proposition 1.7 of \cite{hazama}, one concludes that $mt(A\otimes B)$ contains one copy of
$so(q)$ which maps onto $mt(H)$ and whose projections to $mt(A)$ and to $mt(B)$ are injective.
The Hodge structures on $H$ and the sub-Hodge structures of $A$ and $B$ defined by this copy of $so(q)$ in $mt(A)\times mt(B)$ are obtained from one map of the Lie algebra of $\mathbb{S}^1$ to $so(q)_\mathbb{R}$.

Now one considers the classification of the cases where the
complex Lie algebra $so(q)_\mathbb{C}$  is a (simple) factor of the
complexified  Lie algebra of the  Mumford-Tate group of a weight $1$
polarized Hodge structure $A$ and then one finds all the possible
representations of $so(q)_\mathbb{C}$ on $A_\mathbb{C}$.
This was done by Deligne \cite{deligne2}.

The case where ${\rm dim}\, H$ is odd is the easiest one: in that
case the Lie algebra $so(q)_\mathbb{C}$ has a unique such
representation, which is the spin representation. This spin
representation also occurs on $H^1(A_{KS}(H),\mathbb{C})$,
with the same map of the Lie algebra of $\mathbb{S}^1$ to $so(q)_\mathbb{R}$.
Thus there is non-trivial $so(q)_\mathbb{C}$-equivariant map,
respecting the Hodge structures, from
$A_\mathbb{C}$ to $H^1(A_{KS}(H),\mathbb{C})$.
As the complex vector space of such maps is the complexification of the rational vector space of $so(q)$-equivariant maps from $A$ to $H^1(A_{KS}(H),\mathbb{Q})$,
there is such a map from $A$ to $H^1(A_{KS}(H),\mathbb{Q})$.
It follows that $A$ is a simple factor of the Hodge structure on $H^1(A_{KS}(H),\mathbb{Q})$.

In the case where ${\rm dim}\, H $ is even, the representations of
$so(q)_\mathbb{C}$ that can occur are the standard representation
and the two half spin representations.
However, the tensor product of the standard representation with any
of these three cannot have a subrepresentation which is again the
standard representation. Thus $H$ cannot be a summand of $A\otimes
B$ if $A_\mathbb{C}$ is the standard
representation of $so(q)_\mathbb{C}$. Therefore $A_\mathbb{C}$ must
have a half-spin representation of $so(q)_\mathbb{C}$ as summand. As before,
it follows that $A$ is a summand of the $H^1$ of the
Kuga-Satake variety of $H$.
\end{proof}

\section{Proof of the theorems }
We first prove that Theorem \ref{theomatsu} is a consequence of Theorem
\ref{theomatsuMT}.
Let $X$ be a projective hyper-K\"ahler manifold of dimension
 $2n$ with a Lagrangian fibration $f:X\rightarrow B$.
Then  by Lemma \ref{ledefMT}, there exists a point  (in fact many!) in the   space
$\mathcal{M}_P$ of deformations of
$X$ with constant Picard group which paramereterizes a projective hyper-K\"ahler
manifold
$X'$ such  that
 ${\rm NS}(X')={\rm NS}(X)$ and the Mumford-Tate group of the Hodge structure on
$H^2(X',\mathbb{Q})$ is the orthogonal group of
$(H^2(X',\mathbb{Q})_{tr},q)=(H^2(X,\mathbb{Q})_{tr},q)$. As we assumed that $b_{2,tr}(X)\geq 5$, the same
holds for $X'$. Hence Theorem \ref{theomatsuMT} applies to $X'$, which proves Theorem \ref{theomatsu}.

We now assume that $X=X'$ satisfies the  assumption in Theorem \ref{theomatsuMT} and
turn to the proof of Theorem \ref{theomatsuMT}.
\begin{proof}[Proof of Theorem \ref{theomatsuMT}] Let  $ f:X\rightarrow B$
be a Lagrangian
fibration with ${\rm dim}\,H^2(X,\mathbb{Q})_{tr}\geq5$, $B$ smooth and
$MT(H^2(X,\mathbb{Q})_{tr})=SO(H^2(X,\mathbb{Q})_{tr},q)$. We have to prove
 that $f$ satisfies Matsushita's conjecture, that is, if the general fiber
 of the moduli map $m$ is positive dimensional, then the moduli map is constant.
Let $b\in B$ be a general point and assume the fiber $F_b$ of the
moduli map $m$ passing through $b$ is positive dimensional. Over the
Zariski open set $U=F_b\cap B^0$ of $F_b$,  the Lagrangian
fibration restricts to an isotrivial fibration $X_U\rightarrow U$.
As we are in the projective setting, it follows that after passing
to a generically finite cover $U'$ of $U$, the base-changed family
$X_{U'}\rightarrow U'$ splits as a product $J_b\times U'$, where
the abelian variety $J_b$ is the typical fiber $f^{-1}(b)$, for $b\in U$. Let $F'_b$ be a
smooth projective completion of $U'$ and $X_{F'_b}$ be a smooth
projective completion of $X_{U'}$. The  natural rational map
$X_{F'_b}\dashrightarrow X$ induces a rational map $f_b:J_b\times
F'_b\dashrightarrow X$. Consider the induced morphism of Hodge
structures
$$f_b^*:H^2(X,\mathbb{Q})\rightarrow H^2(J_b\times F'_b,\mathbb{Q}).$$
We claim that the composite map
$$\alpha:H^2(X,\mathbb{Q})\rightarrow H^2(J_b\times F'_b,\mathbb{Q})\rightarrow H^1(J_b,\mathbb{Q})\otimes
H^1(F'_b,\mathbb{Q}),$$
where the second map is given by K\"unneth
decomposition,
has an injective restriction to $H^2(X,\mathbb{Q})_{tr}$.

This indeed follows from the following facts :

a) The Hodge structure on $H^2(X,\mathbb{Q})_{tr}$ is simple. Indeed, it is polarized with
$h^{2,0}$-number equal to $1$
 and
it does not contain nonzero Hodge classes. Hence if there is a nontrivial
sub-Hodge structure $H\subset H^2(X,\mathbb{Q})_{tr}$, it must have $H^{2,0}\not=0$. But then
the orthogonal complement $H^{\perp}\subset H^2(X,\mathbb{Q})_{tr}$ is either trivial or with
nonzero $(2,0)$-part, which contradicts the fact that $ H^{2,0}(X)$ is of dimension $1$.

 b) The $(2,0)$-form
$\sigma$ on $X$ has a nonzero image in $H^0(\Omega_{J_b})\otimes H^0(\Omega_{F'_b})$.
To see this last point, we recall that $J_b$ is Lagrangian, that is, the form $\sigma$ restricts to zero
on $J_b$. If it vanished also in $H^0(\Omega_{J_b})\otimes H^0(\Omega_{F'_b})$, its pull-back
to  $J_b\times F'_b$ would  lie
in $H^0(\Omega_{F'_b}^2)$. But as ${\rm dim}\,F_b>0$, this contradicts the fact that $\sigma$
 is nondegenerate and
${\rm dim}\,J_b=n=\frac{1}{2}{\rm dim}\,X$.
This proves  the claim since by b), the map $\alpha$ is nonzero and thus by a) it is injective.

The abelian variety $J_b$ might not be a simple abelian variety, (or equivalently, the
weight $1$ Hodge structure on $H^1(J_b,\mathbb{Q})$ might not be
simple), but the(polarized)  Hodge structure on $H^1(J_b,\mathbb{Q})$ is a
direct sum of simple weight $1$ Hodge structures
$$H^1(J_b,\mathbb{Q})\cong A_1\oplus\ldots\oplus A_s,$$
and for some $i\in\{1,\ldots s\}$ the induced morphism of Hodge
structures
$$\beta:H^2(X,\mathbb{Q})_{tr}\stackrel{\alpha}{\rightarrow} H^1(J_b,\mathbb{Q})\otimes
H^1(F'_b,\mathbb{Q})\rightarrow A_i\otimes H^1(F'_b,\mathbb{Q})$$
must be nonzero, hence  again injective by the simplicity of the Hodge
structure on $H^2(X,\mathbb{Q})_{tr}$.

 We are now in position to apply Proposition
\ref{propvgeemen} because $A_i$ is simple. We thus conclude that
$A_i$ is isomorphic to a direct summand of $H^1(A_{KS}(X),\mathbb{Q})$,
 where $A_{KS}(X)$ is the Kuga-Satake variety built on the Hodge
structure on $H^2(X,\mathbb{Q})_{tr}$. Equivalently, the abelian
variety $J_b$ contains a nontrivial abelian variety $T_b$ which is
isogenous to  an abelian subvariety of $K(X)$. As there are
finitely many isogeny classes of abelian subvarieties of $A_{KS}(X)$, we
conclude that $T_b$ in fact does not depend on the general point
$b$. Let us call $T$ this abelian subvariety of $A_{KS}(X)$. We now do the following:
for the general point $b\in B$, let $J'_b\subset J_b$ be the sum of
the abelian subvarieties of $J_b$ which are isogenous to $T$. (We
are allowed to do this because $X$ is projective, hence admits a
multisection, hence is isogenous to the associated Jacobian
fibration.) Over a Zariski open set $V$ of $B$, the subvarieties
$J'_b$ vary nicely in family, providing a sub-abelian fibration
$\mathcal{T}\subset X_V$. Using an ample line bundle on $X$, the
fibers $J_b,\,b\in V,$ then split canonically up to isogeny as a
direct sum
$$J_b\stackrel{isog}{\cong} J'_b\oplus J''_b,$$
and again the subvarieties $J''_b$ vary nicely in family, providing
a sub-abelian fibration $\mathcal{S}\subset X_V$. We then have an
isogeny
$$\mu:X_V\rightarrow \mathcal{S}\times_V\mathcal{T}.$$
We know that ${\rm dim}\,\mathcal{T}/V>0$.  If ${\rm
dim}\,\mathcal{S}/V>0$, then we get a contradiction as follows: We
know by \cite{oguiso} that ${\rm NS}\,(X_V/V)=\mathbb{Z}$. But if
$L$ is an ample line bundle on $X$, the pull-backs
$\mu^*L_{\mid\mathcal{T}}$ and $\mu^*L_{\mid \mathcal{S}}$ provide two linearly
independent divisor classes in ${\rm NS}\,(X_V/V)$. Hence we proved that ${\rm
dim}\,\mathcal{S}/V=0$, or equivalently $X_V=\mathcal{T}$. By
construction, $\mathcal{T}\rightarrow V$ is an isotrivial fibration,
so we proved that if the moduli map $m$ has positive dimensional
general fibers, then the fibration is isotrivial.

\end{proof}
\begin{rema}{\rm One may wonder if the hypothesis that $X$ is projective has really been used in
 the proof of Theorem \ref{theomatsuMT}. Indeed, even if $X$ is not projective, one knows that the
 fibers of a Lagrangian fibration are abelian varieties, and even canonically polarized abelian
 varieties. One has to be prudent however, because if the relative
 polarizations do not come from a line bundle
 on the total space $X$ but just form an integral degree $2$ cohomology class which is
 of type $(1,1)$ along the fibers, they do  not allow us to construct
 holomorphic multisections (which extends analytically over the singular fibers), and similarly
 for the relative splitting of the fibration. In the K\"ahler case,
  one can easily make $X$ projective by a small
 deformation preserving a given Lagrangian fibration, so it seems much safer to work with this assumption.
}
\end{rema}
\section{An example}
We construct in this section an example of a projective $K3$ surface
$S$, such that the Hodge structure $H$ on $H^2(S,\mathbb{Q})_{tr}$ can be
realized as a sub-Hodge structure of a tensor product $H_1\otimes H_2$, with
$H_1$ and $H_2$ of weight $1$, for a continuous family of weight $1$ polarized
Hodge structures $H_1$.

We start with a projective $K3$ surface
$S$ admitting a non-symplectic automorphism
$\phi$
of prime order $p\geq 5$ (see \cite{asartaki}, \cite{ogzhang} for construction and classification).
Let
$H=H^2(S,\mathbb{Q})_{prim}$.
\begin{prop}  There is a continuous family of polarized Hodge structures $H_1$
of weight $1$
such that
for some weight $1$ Hodge structure $H_2$, one has
$$H\subset H_1\otimes H_2$$ as Hodge structures.
\end{prop}
\begin{proof}
Let $\lambda\not=1$ be the eigenvalue of $\psi=\phi^*$ acting on $H^{2,0}(S)$.
Let $H_1$ be  any weight $1$ polarized Hodge structure
admitting an automorphism $\psi'$ of order $p$ such that
\begin{enumerate}
\item \label{it1}
$\lambda^{-1}$ is not an eigenvalue of $\psi'$ acting on
$H_1^{1,0}$.

\item\label{it2}  $\lambda^{-1}$ is an eigenvalue of $\psi'$
acting on $H_1^{0,1}$.
\end{enumerate}
For such $H_1$, we find that
the weight $3$ Hodge structure
$$
H_2:=(H_1\otimes H)^{G},
$$
where $G$ is $\mathbb{Z}/p\mathbb{Z}$ acting on $H\otimes H_1$ via
$\psi\otimes \psi'$, is the Tate twist of a weight $1$ Hodge structure
$H_2$, since we have
$$
((H_1\otimes H)^{G})^{3,0}=(H_1^{1,0}\otimes H^{2,0})^G= (H_1^{1,0})^{\lambda^{-1}}\otimes H^{2,0}=0.
$$
On the other hand, $H_2$ is nonzero, since $\lambda^{-1}$ is an eigenvalue of $\psi'$
acting on $H_1^{0,1}$, which by the same argument as above provides a nonzero element
in $(H_1^{0,1}\otimes H_2^{2,0})^G$.

By composing the inclusion
$H_1^*\otimes H_2\hookrightarrow H_1^*\otimes H_1\otimes H$ with the contraction map
%, tensored with the identity on $H$,
$H_1^*\otimes H_1\rightarrow {\mathbb Q}$,
we get a map $H_1^*\otimes H_2\rightarrow H$.
This map is non-trivial, since choosing nonzero
$\sigma\in (H^{0,1}_1)^{\lambda^{-1}}$ and $\eta\in H^{2,0}$
we have $\sigma\otimes\eta=i(\omega)$ for some $\omega \in H_2$.
Next we choose $u\in H_1^*$ such that $u(\sigma)\neq 0$,
then we see that, after tensoring with ${\mathbb C}$,
$u\otimes \omega\mapsto u\otimes\sigma\otimes \eta\mapsto
u(\sigma)\eta\neq 0$.

Since these Hodge structures are polarized, they are isomorphic to their duals up to Tate twists. Thus there is a nontrivial
morphism of Hodge structures
$$H\rightarrow H_1^*\otimes H_2$$
that is injective by the simplicity of the Hodge structure $H$.

We conclude observing that  by the assumption $p\geq 5$, the family of weight $1$ polarized
Hodge structures $H_1$ satisfying conditions \ref{it1} and \ref{it2} above has positive dimension.
\end{proof}

Bert van Geemen

Dipartimento di Matematica

Universit\`a di Milano

Via Saldini 50

20133 Milano

Italia

\smallskip
lambertus.vangeemen@unimi.it

\smallskip\smallskip

Claire Voisin

Institut de Math\'ematiques de Jussieu

 4 place Jussieu

Case 247

75252 Paris Cedex 05

 France

\smallskip
 claire.voisin@imj-prg.fr
    \end{document}